\documentclass[11pt]{article}
\usepackage{epsf}
\usepackage{amsmath}
\usepackage{epsfig}
\usepackage{times}
\usepackage{amssymb}
\usepackage{amsthm}
\usepackage{setspace}
\usepackage{cite}

\usepackage{algorithmic}  
\usepackage{algorithm}

\usepackage{shadow}
\usepackage{fancybox}
\usepackage{fancyhdr}

\def\y{{\bf y}}

\def\x{{\bf x}}

\def\x{{\mathbf x}}

\def\x{{\bf x}}
\def\y{{\bf y}}

\def\h{{\bf h}}

\def\cH{{\cal H}}

\def\be{\begin{equation}}
\def\ee{\end{equation}}
\def\ba{\left[\begin{array}}
\def\ea{\end{array}\right]}

\def\x{{\bf x}}
\def\y{{\bf y}}

\def\1{{\bf 1}}

\def\g{{\bf g}}
\def\0{{\bf 0}}

\newtheorem{theorem}{Theorem}

\newtheorem{lemma}{Lemma}

\begin{document}

\begin{singlespace}

\title {Asymmetric Little model and its ground state energies 
}
\author{
\textsc{Mihailo Stojnic}
\\
\\
{School of Industrial Engineering}\\
{Purdue University, West Lafayette, IN 47907} \\
{e-mail: {\tt mstojnic@purdue.edu}} }
\date{}
\maketitle

\centerline{{\bf Abstract}} \vspace*{0.1in}

In this paper we look at a class of random optimization problems that arise in the forms typically known in statistical physics as Little models. In \cite{BruParRit92} the Little models were studied by means of the well known tool from the statistical physics called the replica theory. A careful consideration produced a physically sound conclusion that the behavior of almost all important features of the Little models essentially resembles the behavior of the corresponding ones of appropriately scaled Sherrington-Kirkpatrick (SK) model. In this paper we revisit the Little models and consider their ground state energies as one of their main features. We then rigorously show that they indeed can be lower-bounded by the corresponding ones related to the SK model. We also provide a mathematically rigorous way to show that the replica symmetric estimate of the ground state energy is in fact a rigorous upper-bound of the ground state energy. Moreover, we then recall on a set of powerful mechanisms we recently designed for a study of the Hopfield models in \cite{StojnicHopBnds10,StojnicMoreSophHopBnds10} and show how one can utilize them to substantially lower the upper-bound that the replica symmetry theory provides.

\vspace*{0.25in} \noindent {\bf Index Terms: Little models; ground-state energy}.

\end{singlespace}

\section{Introduction}
\label{sec:back}

We start by looking at what is typically known in mathematical physics as Little model. The model was popularized in \cite{Little74}. It essentially looks at what is called Hamiltonian of the following type
\begin{equation}
\cH(H,\x,\y)=\sum_{i=1}^{n}\sum_{j=1}^{m}  A_{ij}\x_i\y_j,\label{eq:ham}
\end{equation}
where
\begin{equation}
A_{ij}(H)=H_{ij},\label{eq:hamAij}
\end{equation}
are the so-called quenched interactions and $H$ is an $m\times n$ matrix (we will typically consider scenario where $m$ and $n$ are large and $\frac{m}{n}=\alpha$ where $\alpha$ is a constant independent of $n$; however, many of our results will hold even for fixed $m$ and $n$). Typically, one assumes that the elements of matrix $H$ as well as the elements of vector $\x$ are binary. In fact, to be a bit more precise, one assumes that the elements of $H$ are from set $\{-1,1\}$, whereas the elements of $\x$ are from set $\{-\frac{1}{\sqrt{n}},\frac{1}{\sqrt{n}}\}$ and the elements of $\y$ are from set $\{-\frac{1}{\sqrt{m}},\frac{1}{\sqrt{m}}\}$ (if of this form the elements of $\x$ and $\y$ are typically called binary spins). This is fairly often written a bit differently with scaling factors $\frac{1}{\sqrt{m}}$ and $\frac{1}{\sqrt{n}}$ often taken outside the sum and $\x,\y\in\{-1,1\}$. However, we will find our presentation to be substantially less cumbersome if we do keep the scaling factors inside the sum.

Also, while the assumption that elements of $H$ are binary as well is not strange/novel in statistical physics it is in fact even more common in neural networks communities (see, e.g. \cite{Tal98,ShchTir93,BarGenGueTan10,BarGenGueTan12}). Additionally, in the physics literature one usually follows a convention and introduces a minus sign in front of the Hamiltonian given in (\ref{eq:ham}). Since our main concern is not really the physical interpretation of the given Hamiltonian but rather mathematical properties of such forms we will avoid the minus sign and keep the form as in (\ref{eq:ham}). Also, fairly often one considers the scenario where the interactions on the main diagonal are excluded. This is a bit more typical in say Sherrington-Kirkpatrick (SK) or Hopfield models (see, e.g. \cite{Hop82,SheKir72,StojnicHopBnds10,StojnicMoreSophHopBnds10}) where $m=n$ and/or $\x_i=\y_i,1\leq i\leq n$. Since we consider a bit more general (rather different) scenario we will refrain from such an exclusion. If of interest of course it is relatively easy to remove them (however, we find writing substantially more compact if we keep these terms as well). Also, while on the subject that involves $m=n$ we should mention that when $m=n$ one can consider the scenario with the quenched interactions being symmetric, i.e. $A_{ij}=A_{ji},\forall i,j$ (or alternatively $H_{ij}=H_{ji},\forall i,j$). In that case the Little model is typically referred to as the symmetric Little model (along the same lines the model that we will pursue here is often called the asymmetric Little model). While we defer a thorough discussion about the symmetric model to a companion paper, we briefly mention that in \cite{BruParRit92} the authors considered the asymmetric model but did mention that all results they obtained through their considerations should be translatable to the symmetric case as well (as it will turn out many of the results that we will present below for the asymmetric case can indeed be translated to the symmetric case as well).

To characterize the behavior of physical interpretations that can be described through the above Hamiltonian one then looks at the partition function
\begin{equation}
Z(\beta,H)=\sum_{\x\in\{-\frac{1}{\sqrt{n}},\frac{1}{\sqrt{n}}\}^n,\y\in\{-\frac{1}{\sqrt{m}},\frac{1}{\sqrt{m}}\}^m}e^{\beta\cH(H,\x,\y)},\label{eq:partfun}
\end{equation}
where $\beta>0$ is what is typically called the inverse temperature. Depending on what is the interest of studying one can then also look at a more appropriate scaled $\log$ version of $Z(\beta,H)$ (typically called the free energy)
\begin{equation}
f_p(n,\beta,H)=\frac{\log{(Z(\beta,H)})}{\beta \sqrt{n}}=\frac{\log{(\sum_{\x\in\{-\frac{1}{\sqrt{n}},\frac{1}{\sqrt{n}}\}^n,\y\in\{-\frac{1}{\sqrt{m}},\frac{1}{\sqrt{m}}\}^m}e^{\beta\cH(H,\x,\y)})}}{\beta \sqrt{n}}.\label{eq:logpartfun}
\end{equation}
Studying behavior of the partition function or the free energy of the Little model of course has a long history. Since we will not focus on the entire function in this paper we just briefly mention that a long line of results can be found in e.g. excellent references \cite{BruParRit92,BarGenGue11bip,AmiGutSom85}. In this paper we will focus on studying optimization/algorithmic aspects of $\frac{\log{(Z(\beta,H)})}{\beta n}$. More specifically, we will look at a particular regime $\beta,n\rightarrow\infty$ (which is typically called a zero-temperature thermodynamic limit regime or as we will occasionally call it the ground state regime). In such a regime one has
\begin{eqnarray}
\hspace{-.3in}\lim_{\beta,n\rightarrow\infty}f_p(n,\beta,H)=
\lim_{\beta,n\rightarrow\infty}\frac{\log{(Z(\beta,H)})}{\beta \sqrt{n}}& = &\lim_{n\rightarrow\infty}\frac{\max_{\x\in\{-\frac{1}{\sqrt{n}},\frac{1}{\sqrt{n}}\}^n,\y\in\{-\frac{1}{\sqrt{m}},\frac{1}{\sqrt{m}}\}^m}\cH(H,\x,\y)}{\sqrt{n}}\nonumber \\
& = & \lim_{n\rightarrow\infty}\frac{\max_{\x\in\{-\frac{1}{\sqrt{n}},\frac{1}{\sqrt{n}}\}^n,\y\in\{-\frac{1}{\sqrt{m}},\frac{1}{\sqrt{m}}\}^m}\y^T H\x}{\sqrt{n}},\label{eq:limlogpartfun}
\end{eqnarray}
which essentially renders the following form (often called the ground state energy)
\begin{equation}
\lim_{\beta,n\rightarrow\infty}f_p(n,\beta,H)=\lim_{n\rightarrow\infty}\frac{\max_{\x,\y\in\{-\frac{1}{\sqrt{n}},\frac{1}{\sqrt{n}}\}^n}\y^T H\x}{\sqrt{n}}.\label{eq:posham}
\end{equation}
The above form will be one of the main subjects that we study in this paper. We will refer to the optimization part of (\ref{eq:posham}) as the asymmetric Little form.

As mentioned earlier one can study its a symmetric counterpart as well. Also, as mentioned earlier we will present results in that direction in a forthcoming paper. Additionally, looking at the asymmetric Little form from a purely mathematical point of view, one can also consider the minmax version of the problem. While this problem may not typically be of so much interest from the theoretical physics' point of view it is mathematically a very interesting problem and in fact a generalized version of several hard combinatorial problems. Since, the concepts that we will develop in this paper will be powerful enough to handle that case as well, we will in the later parts of the paper provide a few pointers as to how it can be handled.

For mathematical completeness we will formally define a minmax counterpart of (\ref{eq:posham})
\begin{equation}
\lim_{\beta,n\rightarrow\infty}f_n(n,\beta,H)=\lim_{n\rightarrow\infty}\frac{\min_{\x\in\{-\frac{1}{\sqrt{n}},\frac{1}{\sqrt{n}}\}^n}
\max_{\y\in\{-\frac{1}{\sqrt{m}},\frac{1}{\sqrt{m}}\}^m}\y^T H\x}{\sqrt{n}}.\label{eq:negham}
\end{equation}
We will then correspondingly refer to the optimization part of (\ref{eq:negham}) as the minmax asymmetric Little form (along the same lines, we will sometimes refer to the optimization part of (\ref{eq:posham}) as the max asymmetric Little form).

In the following sections we will present a collection of results that relate to behavior of the forms given in (\ref{eq:posham}) and (\ref{eq:negham}) when they are viewed in a statistical scenario. The results related to (\ref{eq:posham}) that we will present below will essentially correspond to what is called the ground state energy of the asymmetric Little model. As it will turn out, in the statistical scenario that we will consider, (\ref{eq:posham}) and (\ref{eq:negham}) will be almost completely characterized by their corresponding average values
\begin{equation}
\lim_{\beta,n\rightarrow\infty}Ef_p(n,\beta,H)=\lim_{n\rightarrow\infty}\frac{E\max_{\x\in\{-\frac{1}{\sqrt{n}},\frac{1}{\sqrt{n}}\}^n,
\y\in\{-\frac{1}{\sqrt{m}},\frac{1}{\sqrt{m}}\}^m}\y^T H\x}{\sqrt{n}}\label{eq:poshamavg}
\end{equation}
and
\begin{equation}
\lim_{\beta,n\rightarrow\infty}Ef_n(n,\beta,H)=\lim_{n\rightarrow\infty}\frac{E\min_{\x\in\{-\frac{1}{\sqrt{n}},\frac{1}{\sqrt{n}}\}^n}
\max_{\y\in\{-\frac{1}{\sqrt{m}},\frac{1}{\sqrt{m}}\}^m}\y^T H\x}{\sqrt{n}}.\label{eq:neghamavg}
\end{equation}

Also, before proceeding further we should mention a few known/conjectured facts about the asymmetric/symmetric Little models as well as where among many of them our work falls in. As we mentioned earlier, these models have been studied for a long time and instead of reviewing the entire chronology of these studies we will focus on the results that are probably the most relevant when it comes to what we will present below. Also, we preface this short discussion by saying that almost all results mentioned below treat a general Little form, i.e. they treat the free energy and all other parameters of interest at any temperature (i.e. for any value of $\beta$). Of course, this is substantially wider than what we are concerned here with. Our concern is essentially a special case, called the zero-temperature regime where $\beta\rightarrow\infty$. However, all our results can be extended to general partition function and to the free energy at any temperature. Since we are mostly interested in an optimization/algorithmic point of view (rather than a pure physicist point of view) we found it to be substantially easier to present the results only for ``the ground-state" regime (we will present the full blown partition function/free energy analysis elsewhere).

Now, going back to a more concrete discussion about known related results about the Little model. We start by mentioning that an early approach to connect the Little models and the SK model followed what was presented in \cite{CabMarPaoPar88,AmiGutSom85}. In \cite{AmiGutSom85} considered an application of the Little model in a neural network context. They employed the replica theory approach and assuming a replica symmetry they obtained a same set of results one would get for an appropriately scaled SK model. In \cite{CabMarPaoPar88} the authors conducted a various numerical experiments that eventually suggested that the symmetric Little model may be a bit more different from the SK model than initially thought. Namely, they obtained values for the free energy at certain temperatures that seemed to suggest that if one is to extrapolate them to a thermodynamic limit regime that the energies would be a bit higher than the corresponding ones of the SK model (we should emphasize that since we keep avoiding the negative sign of some of the quantities this phenomenon is termed slightly differently in \cite{CabMarPaoPar88,BruParRit92}; namely, according to the definitions of the models overthere the energies were considered as potentially a bit lower than the corresponding ones of the SK model). However, \cite{BruParRit92} then went a bit further and attempted to provide a sound settling of this phenomenon. Namely, in \cite{BruParRit92} the authors recalled that all numerical experiments are typically done on finite size examples. Then they further pointed out that deducing the thermodynamic behavior based on that may not always work correctly (moreover, they also pointed out that if one works with a finite $n$, occasionally there could even be a substantial difference between the results one would get for say Gaussian or Bernoulli Little models; we will say more on the definitions of these below).

Continuing further along these lines the authors in \cite{BruParRit92} then proceeded with the analysis of the asymmetric Little model through the replica framework. Eventually their analysis produced the same answer one initially could have predicted, i.e. the thermodynamic behavior of the asymmetric Little model should indeed match the thermodynamic behavior of the appropriately scaled SK model. Moreover, while the authors in \cite{BruParRit92} worked with the asymmetric Little model (in the framework that they used it offers a nice set of numerical simplifications) they also pointed out that it is expected that all the results that they presented should be applicable to the symmetric Little model.

In a piece of a bit more recent work the authors in \cite{BarGenGue11bip} considered a set of models that also includes the asymmetric Little model. They presented a collection of rigorous as well as replica type results. In fact, they were able to show that in certain low $\beta$ regime (or in a physicist terminology, in a certain high temperature regime) the two actually match. Moreover, they also showed several replica symmetric predictions as to how the free energy should behave in the entire range of $\beta$'s. Also, as they correctly pointed out, it turns out that the replica symmetry predictions are not sustainable as $\beta$ grows and consequently will not be correct for $\beta$'s larger than a critical value (and certainly not for $\beta\rightarrow\infty$). In the regime where the replica symmetry predictions are correct \cite{BarGenGue11bip} does show that the behavior of the asymmetric Little model indeed matches the corresponding one of the SK model.

As mentioned earlier, in this paper we attack probably the hardest regime, i.e. the zero-temperature -- $\beta\rightarrow\infty$ regime. (Somewhat contradictory, while this regime is long believed to be the hardest for achieving a mathematically sound resolution, it is the easiest to present which, as we mentioned earlier, is the reason why we chose it!) The results that we will present hint that even in such a regime it is reasonable to believe that the replica theory predictions could be correct. While the symmetry may need to be broken, the original wisdom that the Little models should resemble the SK model may still be in place (of course, as is now well known in the zero-temperature regime the symmetry has to be broken in the SK model as well).

Before proceeding further with our presentation we will briefly discuss the organization of the paper. In Section \ref{sec:poshop} we will consider the asymmetric Little form. More specifically, in Subsection \ref{sec:poshoplb} we will present a way to create a lower bound on the the ground state energy of the asymmetric Little model. In Subsection \ref{sec:poshopub} we will present a way to create an upper bound on the very same ground state energy. Moreover, the upper bound that we will present in Subsection \ref{sec:poshopub} will match the one that can be obtained through the replica approach (assuming the replica symmetry). In Subsection \ref{sec:poshopublow} we will then present a mechanism that can ``substantially" lower the upper bounds from Subsection \ref{sec:poshopub}. In Section \ref{sec:neghop} we will consider the minmax counterpart of the asymmetric Little form. Along the same lines, we will then in Section \ref{sec:neghop} present a way that can create the minmax counterparts for many of the results from Section \ref{sec:poshop}. Finally in Section \ref{sec:conc} we will present a few concluding remarks.

\section{Asymmetric Little form}
\label{sec:poshop}

In this section we will look at the following optimization problem (which clearly is the key component in estimating the ground state energy of the asymmetric Little model in the thermodynamic limit)
\begin{equation}
\max_{\x\in\{-\frac{1}{\sqrt{n}},\frac{1}{\sqrt{n}}\}^n,\y\in\{-\frac{1}{\sqrt{m}},\frac{1}{\sqrt{m}}\}^m}\y^T H\x.\label{eq:posham1}
\end{equation}
For a deterministic (given fixed) $H$ this problem is of course known to be NP-hard (it essentially falls into the class of binary bilinear optimization problems). Instead of looking at the problem in (\ref{eq:posham1}) in a deterministic way i.e. in a way that assumes that matrix $H$ is deterministic, we will look at it in a statistical scenario (this is of course a typical scenario in statistical physics). Within a framework of statistical physics (and even more so within the framework of neural networks) the problem in (\ref{eq:posham1}) (or say similar ones like the Hopfield form) is studied assuming that elements of matrix $H$ are comprised of Bernoulli $\{-1,1\}$ i.i.d. random variables see, e.g. \cite{Tal98,PasShchTir94,BruParRit92,AmiGutSom85}. While our results will turn out to hold in such a scenario as well we will present them in a different scenario: namely, we will assume that the elements of matrix $H$ are i.i.d. standard normals. We will then call the form (\ref{eq:posham1}) with Gaussian $H$, the Gaussian asymmetric Little form. On the other hand, we will call the form (\ref{eq:posham1}) with Bernoulli $H$, the Bernoulli asymmetric Little form. In the remainder of this section we will look at possible ways to estimate the optimal value of the optimization problem in (\ref{eq:posham1}). In the first part below we will introduce a strategy that can be used to obtain a lower bound on the optimal value. In the second part we will then create a corresponding upper-bounding strategy. As it will turn out the lower bound will match the expected prediction that relates the asymmetric Little model to the SK-model. However, the upper bound will be a bit above that. In the third part of this section we will then present a mechanism that can be used to lower the upper-bound. Finally the upper-bound will come fairly close to matching the lower bound (which, as mentioned above, essentially matches the result obtained in \cite{BruParRit92} and concurs with the corresponding (appropriately scaled) one for the SK model). This hints that the prediction made in \cite{BruParRit92} on the grounds of the replica theory may in fact be mathematically rigorously correct.

\subsection{Lower-bounding ground state energy of the asymmetric Little form}
\label{sec:poshoplb}

In this section we will look at the optimization problem from (\ref{eq:posham1}). In fact, we will look at its objective value. To that end let us introduce a quantity $\xi_p$ as its objective value
\begin{equation}
\xi_p=\max_{\x\in\{-\frac{1}{\sqrt{n}},\frac{1}{\sqrt{n}}\}^n,\y\in\{-\frac{1}{\sqrt{m}},\frac{1}{\sqrt{m}}\}^m}\y^T H\x.\label{eq:sqrtposham1}
\end{equation}
As mentioned above, we will assume that the elements of $H$ are i.i.d. standard normal random variables.

In what follows we will to a degree follow the strategies presented in \cite{StojnicHopBnds10,StojnicMoreSophHopBnds10} when the Hopfield models were discussed. We will try to provide as thorough explanation as possible but will still on occasion assume a substantial level of familiarity with many of the results we presented in \cite{StojnicHopBnds10,StojnicMoreSophHopBnds10}.

As we did in \cite{StojnicHopBnds10,StojnicMoreSophHopBnds10}, we before proceeding further with the analysis of (\ref{eq:sqrtposham1}) recall on several well known results that relate to Gaussian random variables and the processes they create. First we recall on the following results from \cite{Slep62,Gordon88} that relate to statistical properties of certain Gaussian processes.
\begin{theorem}(\cite{Slep62,Gordon88})
\label{thm:Slepian1} Let $X_{i}$ and $Y_{i}$, $1\leq i\leq n$, be two centered Gaussian processes which satisfy the following inequalities for all choices of indices
\begin{enumerate}
\item $E(X_{i}^2)=E(Y_{i}^2)$
\item $E(X_{i}X_{l})\leq E(Y_{i}Y_{l}), i\neq l$.
\end{enumerate}
Then
\begin{equation*}
P(\bigcap_{i}(X_{i}\geq \lambda_{i}))\leq P(\bigcap_{i}(Y_{i}\geq \lambda_{i}))
\Leftrightarrow P(\bigcup_{i}(X_{i}\geq \lambda_{i}))\leq P(\bigcup_{i}(Y_{i}\geq \lambda_{i})).
\end{equation*}
\end{theorem}
The following, more simpler, version of the above theorem relates to the expected values.
\begin{theorem}(\cite{Slep62,Gordon88})
\label{thm:Slepian2} Let $X_{i}$ and $Y_{i}$, $1\leq i\leq n$, be two centered Gaussian processes which satisfy the following inequalities for all choices of indices
\begin{enumerate}
\item $E(X_{i}^2)=E(Y_{i}^2)$
\item $E(X_{i}X_{l})\leq E(Y_{i}Y_{l}), i\neq l$.
\end{enumerate}
Then
\begin{equation*}
E(\min_{i}(X_{i}))\leq E(\min_i(Y_{i})) \Leftrightarrow E(\max_{i}(X_{i}))\geq E(\max_i(Y_{i})).
\end{equation*}
\end{theorem}

Now, to create a lower-bounding strategy for the asymmetric Little form we will rely on Theorems \ref{thm:Slepian1} and Theorem \ref{thm:Slepian2}. We start by reformulating the problem in (\ref{eq:sqrtposham1}) in the following way
\begin{equation}
\xi_p=\max_{\x\in\{-\frac{1}{\sqrt{n}},\frac{1}{\sqrt{n}}\}^n}\max_{\y\in\{-\frac{1}{\sqrt{m}},\frac{1}{\sqrt{m}}\}^m}\y^TH\x.\label{eq:sqrtposham2}
\end{equation}
We will first focus on the expected value of $\xi_p$ and then on its more general probabilistic properties. The following is then a direct application of Theorem \ref{thm:Slepian2}.
\begin{lemma}
Let $H$ be an $m\times n$ matrix with i.i.d. standard normal components. Let $H^{(1)}$ and $H^{(2)}$ be $m\times m$ and $n\times n$ matrices, respectively, with i.i.d. standard normal components. Then
\begin{equation}
E(\max_{\x\in\{-\frac{1}{\sqrt{n}},\frac{1}{\sqrt{n}}\}^n,\y\in\{-\frac{1}{\sqrt{m}},\frac{1}{\sqrt{m}}\}^m}(\y^T H\x))\geq E(\max_{\x\in\{-\frac{1}{\sqrt{n}},\frac{1}{\sqrt{n}}\}^n,\y\in\{-\frac{1}{\sqrt{m}},\frac{1}{\sqrt{m}}\}^m}
(\frac{1}{\sqrt{2}}\y^TH^{(1)}\y+\frac{1}{\sqrt{2}}\x^TH^{(2)}\x)).\label{eq:posexplemmalb}
\end{equation}\label{lemma:posexplemmalb}
\end{lemma}
\begin{proof}
The proof of a similar lemma is presented in \cite{StojnicHopBnds10} and is a direct application of Theorem \ref{thm:Slepian2}. The only thing that is now different is the set of allowed $\y$'s which structurally changes nothing in the proof.
\end{proof}

Using results of Lemma \ref{lemma:posexplemmalb} we then have
\begin{multline}
E(\max_{\x\in\{-\frac{1}{\sqrt{n}},\frac{1}{\sqrt{n}}\}^n,\y\in\{-\frac{1}{\sqrt{m}},\frac{1}{\sqrt{m}}\}^m} \y^T H\x) \geq
E(\max_{\x\in\{-\frac{1}{\sqrt{n}},\frac{1}{\sqrt{n}}\}^n,\y\in\{-\frac{1}{\sqrt{m}},\frac{1}{\sqrt{m}}\}^m}\frac{1}{\sqrt{2}}\y^TH^{(1)}\y+\frac{1}{\sqrt{2}}\x^TH^{(2)}\x)\\
=E(\max_{\y\in\{-\frac{1}{\sqrt{n}},\frac{1}{\sqrt{n}}\}^n}\frac{1}{\sqrt{2}}\y^TH^{(1)}\y)+E(\max_{\x\in\{-\frac{1}{\sqrt{n}},\frac{1}{\sqrt{n}}\}^n}\frac{1}{\sqrt{2}}\x^TH^{(2)}\x),\label{eq:poshopaftlemma2lb}
\end{multline}
and after scaling
\begin{multline}
\frac{E(\max_{\x\in\{-\frac{1}{\sqrt{n}},\frac{1}{\sqrt{n}}\}^n,\y\in\{-\frac{1}{\sqrt{m}},\frac{1}{\sqrt{m}}\}^m}(\y^T H\x))}{\sqrt{n}}\\\geq
\frac{E(\max_{\y\in\{-\frac{1}{\sqrt{m}},\frac{1}{\sqrt{m}}\}^m}(\y^TH^{(1)}\y))}{\sqrt{2n}}+\frac{E(\max_{\x\in\{-\frac{1}{\sqrt{n}},\frac{1}{\sqrt{n}}\}^n}\x^TH^{(2)}\x))}{\sqrt{2n}}.\label{eq:poshopaftlemma3lb}
\end{multline}
Now, using incredible results of \cite{Parisi80,Tal06,Guerra03} one has
\begin{eqnarray}
\lim_{n\rightarrow \infty}\frac{E(\max_{\x\in\{-\frac{1}{\sqrt{n}},\frac{1}{\sqrt{n}}\}^n}\x^TH^{(2)}\x))}{\sqrt{2n}} & = & \xi_{SK}\approx 0.763\nonumber \\
\lim_{n\rightarrow \infty}\frac{E(\max_{\y\in\{-\frac{1}{\sqrt{m}},\frac{1}{\sqrt{m}}\}^n}\y^TH^{(1)}\y))}{\sqrt{2n}} & = & \sqrt{\alpha}\xi_{SK}\approx \sqrt{\alpha}\times 0.763,\label{eq:skmodel}
\end{eqnarray}
where $\xi_{SK}$ is the average ground state energy of the so-called Sherrington-Kirkpatrick (SK) model in the thermodynamic limit (more on the SK model can be found in excellent references \cite{Parisi80,Tal06,Guerra03,SheKir72}). We do mention that the work of \cite{Parisi80,Tal06,Guerra03} indeed settled the thermodynamic behavior of the SK model. However, the characterization of $\xi_{SK}$ in \cite{Parisi80,Tal06,Guerra03} is not explicit and the value we give in (\ref{eq:skmodel}) is a numerical estimate (it is quite likely though, that the estimate we give is a bit conservative; the true value is probably more around $0.7632$). Connecting (\ref{eq:poshopaftlemma3lb}) and (\ref{eq:skmodel}) one then has the following lower-bound of the ground state energy of the asymmetric Little model
\begin{equation}
\lim_{n\rightarrow\infty}\frac{E\xi_p}{\sqrt{n}}=\lim_{n\rightarrow\infty}\frac{E(\max_{\x\in\{-\frac{1}{\sqrt{n}},\frac{1}{\sqrt{n}}\}^n,
\y\in\{-\frac{1}{\sqrt{m}},\frac{1}{\sqrt{m}}\}^m} \y^T H\x)}{\sqrt{n}} \geq \sqrt{\alpha}\xi_{SK}+\xi_{SK}\approx (\sqrt{\alpha}+1)\times 0.763.\label{eq:poshopubexplb}
\end{equation}
The above lower bound (apart from a scaling factor $(\sqrt{\alpha}+1)$) pretty much matches the ground state energy of the SK model. It is also relatively straightforward to see how the scaling factor comes about. In fact there are many ways how it can be done. However, it is neat to have it presented at some point so we will take the opportunity to briefly sketch the idea right here.

Namely, let $H^{(2,s)}$ be a random symmetric matrix (i.e. $H_{ij}^{(2,s)}=H_{ji}^{(2,s)},\forall i,j$) with i.i.d. standard normal components. Typically the ground state energy of the SK model is then defined as
\begin{eqnarray}
\xi_{SK}& = & \lim_{n\rightarrow\infty}\frac{E(\max_{\x\in\{-\frac{1}{\sqrt{n}},\frac{1}{\sqrt{n}}\}^n} \sum_{i=1}^{n}\sum_{j=1,j>i}^{n}H_{ij}^{(2,s)}\x_i\x_j)}{\sqrt{n}}\nonumber \\
& = & \lim_{n\rightarrow\infty}\frac{E(\max_{\x\in\{-\frac{1}{\sqrt{n}},\frac{1}{\sqrt{n}}\}^n} \sum_{i=1}^{n}\sum_{j=1,j\neq i}^n H_{ij}^{(2,s)}\x_i\x_j)}{2\sqrt{n}}.\label{eq:defskgreng}
\end{eqnarray}
This of course comes from the definition of the Hamiltonian that views the quenched interactions only in one direction, or in other words views them symmetrically. Let (as in Lemma \ref{lemma:posexplemmalb}) $H^{(2)}$ be an $n\times n$ matrix with i.i.d. standard normal entries. Then
\begin{eqnarray}
\hspace{-.5in}\lim_{n\rightarrow\infty}\frac{E(\max_{\x\in\{-\frac{1}{\sqrt{n}},\frac{1}{\sqrt{n}}\}^n} \sum_{i=1}^{n}\sum_{j=1}^{n}H_{ij}^{(2)}\x_i\x_j)}{\sqrt{2n}}& = & \lim_{n\rightarrow\infty}\frac{E(\max_{\x\in\{-\frac{1}{\sqrt{n}},\frac{1}{\sqrt{n}}\}^n} \sum_{i=1}^{n}\sum_{j=1,j\neq i}^{n}H_{ij}^{(2)}\x_i\x_j)}{\sqrt{2n}}\nonumber \\
& = & \lim_{n\rightarrow\infty}\frac{E(\max_{\x\in\{-\frac{1}{\sqrt{n}},\frac{1}{\sqrt{n}}\}^n} \sum_{i=1}^{n}\sum_{j=1,j\neq i}^{n}(H_{ij}^{(2)}+(H_{ji}^{(2)}))\x_i\x_j)}{2\sqrt{2n}} \nonumber \\
& = & \lim_{n\rightarrow\infty}\frac{E(\max_{\x\in\{-\frac{1}{\sqrt{n}},\frac{1}{\sqrt{n}}\}^n} \sum_{i=1}^{n}\sum_{j=1,j\neq i}^{n}(H_{ij}^{(2,s)}\x_i\x_j)}{2\sqrt{n}},\label{eq:defskgreng1}
\end{eqnarray}
where statistically (in fact, even literally) speaking one can write
\begin{equation*}
H_{ij}^{(2,s)}=\frac{H_{ij}^{(2)}+H_{ji}^{(2)}}{\sqrt{2}},i\neq j.
\end{equation*}
Combining (\ref{eq:defskgreng}) and (\ref{eq:defskgreng1}) one essentially has
\begin{eqnarray}
\lim_{n\rightarrow\infty}\frac{E(\max_{\x\in\{-\frac{1}{\sqrt{n}},\frac{1}{\sqrt{n}}\}^n} \x^TH^{(2)}\x)}{\sqrt{2n}}=\lim_{n\rightarrow\infty}\frac{E(\max_{\x\in\{-\frac{1}{\sqrt{n}},\frac{1}{\sqrt{n}}\}^n} \sum_{i=1}^{n}\sum_{j=1}^{n}H_{ij}^{(2)}\x_i\x_j)}{\sqrt{2n}}=\xi_{SK},\label{eq:defskgreng2}
\end{eqnarray}
which is exactly what is stated in the first part of (\ref{eq:skmodel}). The second part of (\ref{eq:skmodel}) follows in an analogous fashion. The only difference is possible mismatch between $m$ and $n$ which results in an extra scaling factor of $\sqrt{\alpha}$. The results of (\ref{eq:skmodel}) then easily imply (\ref{eq:poshopubexplb}) as shown above.

We now turn to deriving a more general probabilistic result related to $\xi_p$. We will do so through the following lemma (essentially a direct consequence of Theorem \ref{thm:Slepian1}).
\begin{lemma}
Let $H$ be an $m\times n$ matrix with i.i.d. standard normal components. Let $H^{(1)}$ and $H^{(2)}$ be $m\times m$ and $n\times n$ matrices, respectively, with i.i.d. standard normal components. Let $\zeta$ be a scalar. Then
\begin{equation}
\hspace{-.7in}P(\max_{\x\in\{-\frac{1}{\sqrt{n}},\frac{1}{\sqrt{n}}\}^n,\y\in\{-\frac{1}{\sqrt{m}},\frac{1}{\sqrt{m}}\}^m}(\y^T A\x-\zeta)\geq 0)\geq
P(\max_{\x\in\{-\frac{1}{\sqrt{n}},\frac{1}{\sqrt{n}}\}^n,\y\in\{-\frac{1}{\sqrt{m}},\frac{1}{\sqrt{m}}\}^m}
(\frac{1}{\sqrt{2}}\y^TH^{(1)}\y+\frac{1}{\sqrt{2}}\x^TH^{(2)}\x)-\zeta)\geq 0).\label{eq:posproblemmalb}
\end{equation}\label{lemma:posproblemmalb}
\end{lemma}
\begin{proof}
As mentioned above, the proof is basically a direct consequence of Theorem \ref{thm:Slepian1}.
\end{proof}

Assuming all $\epsilon$'s are arbitrarily small positive constants, setting $\zeta=\xi_p^{(l)}=(1-\epsilon_1^{(m_{sk})})\sqrt{m}+(1-\epsilon_1^{(n_{sk})}\xi_{SK})\sqrt{n}$, and following step by step what was done in \cite{StojnicHopBnds10} (the only difference being the structure of the set of allowed $\y$'s) one arrives at
\begin{multline}
P(\max_{\x\in\{-\frac{1}{\sqrt{n}},\frac{1}{\sqrt{n}}\}^n,\y\in\{-\frac{1}{\sqrt{m}},\frac{1}{\sqrt{m}}\}^m}\y^T H\x\geq \xi_p^{(l)})\\\geq
\lim_{n\rightarrow\infty}P(\max_{\x\in\{-\frac{1}{\sqrt{n}},\frac{1}{\sqrt{n}}\}^n,\y\in\{-\frac{1}{\sqrt{m}},\frac{1}{\sqrt{m}}\}^m}(\frac{1}{\sqrt{2}}\y^TH^{(1)}\y+\frac{1}{\sqrt{2}}\x^TH^{(2)}\x)-\xi_p^{(l)}\geq 0)\geq 1.\label{eq:probanal5lb}
\end{multline}

We summarize our results from this subsection in the following lemma.

\begin{lemma}
Let $H$ be an $m\times n$ matrix with i.i.d. standard normal components. Let $n$ be large and let $m=\alpha n$, where $\alpha>0$ is a constant independent of $n$. Let $\xi_p$ be as in (\ref{eq:sqrtposham1}). Let $\xi_{SK}$ be the average ground state energy in the thermodynamic limit of the SK model as defined in (\ref{eq:skmodel}). Further, let all $\epsilon$'s be arbitrarily small constants independent of $n$ and let $\xi_p^{(l)}$ be a scalar such that
\begin{equation}
\frac{\xi_p^{(l)}}{\sqrt{n}}= (1-\epsilon_1^{(m_{sk})})\sqrt{\alpha}+(1-\epsilon_1^{(n_{sk})})\xi_{SK}.\label{eq:condxipuposgenlemmalb}
\end{equation}
Then
\begin{eqnarray}
& & \lim_{n\rightarrow\infty}P(\max_{\x\in\{-\frac{1}{\sqrt{n}},\frac{1}{\sqrt{n}}\}^n,\y\in\{-\frac{1}{\sqrt{m}},\frac{1}{\sqrt{m}}\}^m}\y^T H\x)\geq \xi_p^{(l)})\geq 1\nonumber \\
& \Leftrightarrow & \lim_{n\rightarrow\infty}P(\xi_p\geq \xi_p^{(l)})\geq 1 \nonumber \\
& \Leftrightarrow & \lim_{n\rightarrow\infty}P(\xi_p^2\geq (\xi_p^{(l)})^2)\geq 1, \label{eq:posgenproblemmalb}
\end{eqnarray}
and
\begin{equation}
\lim_{n\rightarrow\infty}\frac{E\xi_p}{\sqrt{n}}=\lim_{n\rightarrow\infty}\frac{E(\max_{\x\in\{-\frac{1}{\sqrt{n}},\frac{1}{\sqrt{n}}\}^n
,\y\in\{-\frac{1}{\sqrt{m}},\frac{1}{\sqrt{m}}\}^m} \y^T H\x)}{\sqrt{n}} \geq (\sqrt{\alpha}+1)\xi_{SK}\approx (\sqrt{\alpha}+1)\times 0.763.\label{eq:posgenexplemmalb}
\end{equation}
\label{lemma:posgenlemmalb}
\end{lemma}
\begin{proof}
The proof follows from the above discussion, (\ref{eq:poshopubexplb}), and (\ref{eq:probanal5lb}).
\end{proof}

\noindent \textbf{Remark:} One can of course be a bit more specific as to in what way the above probabilities are converging as $n\rightarrow\infty$. However, to make writing as neat as possible we choose probably the simplest way to characterize it.

\subsection{Upper-bounding ground state energy of the asymmetric Little form}
\label{sec:poshopub}

In this subsection we will create the corresponding upper-bound results. To create an upper-bounding strategy for the asymmetric Little form we will again (as in the previous subsection) rely on Theorems \ref{thm:Slepian1} and \ref{thm:Slepian2}. We start by recalling what the problem of interest is. Namely, we will be interested in providing a characterizing of an upper bound on the typical behavior of $\xi_p$
\begin{equation}
\xi_p=\max_{\x\in\{-\frac{1}{\sqrt{n}},\frac{1}{\sqrt{n}}\}^n}\max_{\y\in\{-\frac{1}{\sqrt{m}},\frac{1}{\sqrt{m}}\}^m}\y^TH\x.\label{eq:sqrtposham2lb}
\end{equation}
As in the previous subsection, we will first focus on the expected value of $\xi_p$ and then on its more general probabilistic properties. We will closely follow what was done in the previous subsection and to a degree what was done earlier in \cite{StojnicHopBnds10}. We start by recalling on the following direct application of Theorem \ref{thm:Slepian2}.

\begin{lemma}
Let $H$ be an $m\times n$ matrix with i.i.d. standard normal components. Let $\g$ and $\h$ be $m\times 1$ and $n\times 1$ vectors, respectively, with i.i.d. standard normal components. Also, let $g$ be a standard normal random variable. Then
\begin{equation}
E(\max_{\x\in\{-\frac{1}{\sqrt{n}},\frac{1}{\sqrt{n}}\}^n,\y\in\{-\frac{1}{\sqrt{m}},\frac{1}{\sqrt{m}}\}^m}(\y^T H\x + g))\leq E(\max_{\x\in\{-\frac{1}{\sqrt{n}},\frac{1}{\sqrt{n}}\}^n,\y\in\{-\frac{1}{\sqrt{m}},\frac{1}{\sqrt{m}}\}^m}(\g^T\y+\h^T\x)).\label{eq:posexplemma}
\end{equation}\label{lemma:posexplemma}
\end{lemma}
\begin{proof}
The proof is sketched in \cite{StojnicHopBnds10} for a slightly different setup and follows as a standard/direct application of Theorem \ref{thm:Slepian2}.
\end{proof}

Using results of Lemma \ref{lemma:posexplemma} we then have
\begin{multline}
E(\max_{\x\in\{-\frac{1}{\sqrt{n}},\frac{1}{\sqrt{n}}\}^n,\y\in\{-\frac{1}{\sqrt{m}},\frac{1}{\sqrt{m}}\}^m} \y^T H\x) =E(\max_{\x\in\{-\frac{1}{\sqrt{n}},\frac{1}{\sqrt{n}}\}^n,\y\in\{-\frac{1}{\sqrt{m}},\frac{1}{\sqrt{m}}\}^m}(\y^T H\x +g))\\\leq E(\max_{\x\in\{-\frac{1}{\sqrt{n}},\frac{1}{\sqrt{n}}\}^n,\y\in\{-\frac{1}{\sqrt{m}},\frac{1}{\sqrt{m}}\}^m}(\g^T\y+\h^T\x))=E\sum_{i=1}^{n}|\g_i|+E\sum_{i=1}^{n}|\h_i|\leq \sqrt{\frac{2}{\pi}}\sqrt{m}+\sqrt{\frac{2}{\pi}}\sqrt{n}.\label{eq:poshopaftlemma2}
\end{multline}
Connecting beginning and end of (\ref{eq:poshopaftlemma2}) we finally have an upper bound on $E\xi_p$ from (\ref{eq:sqrtposham1}), i.e.
\begin{equation}
E\xi_p=E(\max_{\x\in\{-\frac{1}{\sqrt{n}},\frac{1}{\sqrt{n}}\}^n,\y\in\{-\frac{1}{\sqrt{m}},\frac{1}{\sqrt{m}}\}^m} \y^T H\x) \leq \sqrt{\frac{2}{\pi}}\sqrt{m}+\sqrt{\frac{2}{\pi}}\sqrt{n}=\sqrt{n}(\sqrt{\alpha}+1)\sqrt{\frac{2}{\pi}},\label{eq:poshopubexp}
\end{equation}
or in a scaled (possibly) more convenient form
\begin{equation}
\frac{E\xi_p}{\sqrt{n}}=\frac{E(\max_{\x\in\{-\frac{1}{\sqrt{n}},\frac{1}{\sqrt{n}}\}^n,\y\in\{-\frac{1}{\sqrt{m}},\frac{1}{\sqrt{m}}\}^m} \y^T H\x)}{\sqrt{n}} \leq (\sqrt{\alpha}+1)\sqrt{\frac{2}{\pi}}.\label{eq:poshopubexp1}
\end{equation}

As in the previous subsection we now turn to deriving a more general probabilistic result related to $\xi_p$. We will do so through the following lemma.
\begin{lemma}
Let $H$ be an $m\times n$ matrix with i.i.d. standard normal components. Let $\g$ and $\h$ be $m\times 1$ and $n\times 1$ vectors, respectively, with i.i.d. standard normal components. Also, let $g$ be a standard normal random variable and let $\zeta_{\x}$ be a function of $\x$. Then
\begin{equation}
P(\max_{\x\in\{-\frac{1}{\sqrt{n}},\frac{1}{\sqrt{n}}\}^n,\y\in\{-\frac{1}{\sqrt{m}},\frac{1}{\sqrt{m}}\}^m}(\y^T A\x + g-\zeta_{\x})\geq 0)\leq
P(\max_{\x\in\{-\frac{1}{\sqrt{n}},\frac{1}{\sqrt{n}}\}^n,\y\in\{-\frac{1}{\sqrt{m}},\frac{1}{\sqrt{m}}\}^m}(\g^T\y+\h^T\x-\zeta_{\x})\geq 0).\label{eq:posproblemma}
\end{equation}\label{lemma:posproblemma}
\end{lemma}
\begin{proof}
The proof is basically same as the proof of Lemma \ref{lemma:posexplemma}. The only difference is that instead of Theorem \ref{thm:Slepian2} it relies on Theorem \ref{thm:Slepian1}.
\end{proof}

Assume all $\epsilon$'s are positive arbitrarily small constants. Let $\zeta_{\x}=-\epsilon_{5}^{(g)}\sqrt{n}+\xi_p^{(u)}$ where $\xi_p^{(u)}$ is such that
\begin{eqnarray}
& & (1+\epsilon_{1}^{(m)})\sqrt{m}\sqrt{\frac{2}{\pi}}+(1+\epsilon_{1}^{(n)})\sqrt{n}\sqrt{\frac{2}{\pi}}+\epsilon_{5}^{(g)}\sqrt{n}<\xi_p^{(u)}\nonumber \\
& \Leftrightarrow & (1+\epsilon_{1}^{(m)})\sqrt{\alpha}\sqrt{\frac{2}{\pi}}+(1+\epsilon_{1}^{(n)})\sqrt{\frac{2}{\pi}}+\epsilon_{5}^{(g)}<\frac{\xi_p^{(u)}}{\sqrt{n}}.\label{eq:condxipu}
\end{eqnarray}
Following what was done in \cite{StojnicHopBnds10} one then has
\begin{multline}
\lim_{n\rightarrow\infty}P(\max_{\x\in\{-\frac{1}{\sqrt{n}},\frac{1}{\sqrt{n}}\}^n,\y\in\{-\frac{1}{\sqrt{m}},\frac{1}{\sqrt{m}}\}^m}(\y^T H\x)\geq \xi_p^{(u)})\\\leq \lim_{n\rightarrow\infty}P(\max_{\x\in\{-\frac{1}{\sqrt{n}},\frac{1}{\sqrt{n}}\}^n,\y\in\{-\frac{1}{\sqrt{m}},\frac{1}{\sqrt{m}}\}^m}
(\g^T\y+\h^T\x+\epsilon_{5}^{(g)}\sqrt{n})\geq \xi_p^{(u)})\leq 0.\label{eq:leftprobanal3}
\end{multline}

We summarize our results from this subsection in the following lemma.

\begin{lemma}
Let $H$ be an $m\times n$ matrix with i.i.d. standard normal components. Let $n$ be large and let $m=\alpha n$, where $\alpha>0$ is a constant independent of $n$. Let $\xi_p$ be as in (\ref{eq:sqrtposham1}). Let all $\epsilon$'s be arbitrarily small constants independent of $n$ and let $\xi_p^{(u)}$ be a scalar such that
\begin{equation}
(1+\epsilon_{1}^{(m)})\sqrt{\alpha}\sqrt{\frac{2}{\pi}}+(1+\epsilon_{1}^{(n)})\sqrt{\frac{2}{\pi}}+\epsilon_{5}^{(g)}<\frac{\xi_p^{(u)}}{\sqrt{n}}.\label{eq:condxipuposgenlemma}
\end{equation}
Then
\begin{eqnarray}
& & \lim_{n\rightarrow\infty}P(\max_{\x\in\{-\frac{1}{\sqrt{n}},\frac{1}{\sqrt{n}}\}^n,\y\in\{-\frac{1}{\sqrt{m}},\frac{1}{\sqrt{m}}\}^m}(\y^T H\x)\leq \xi_p^{(u)})\geq 1\nonumber \\
& \Leftrightarrow & \lim_{n\rightarrow\infty}P(\xi_p\leq \xi_p^{(u)})\geq 1 \nonumber \\
& \Leftrightarrow & \lim_{n\rightarrow\infty}P(\xi_p^2\leq (\xi_p^{(u)})^2)\geq 1, \label{eq:posgenproblemma}
\end{eqnarray}
and
\begin{equation}
\frac{E\xi_p}{\sqrt{n}}=\frac{E(\max_{\x\in\{-\frac{1}{\sqrt{n}},\frac{1}{\sqrt{n}}\}^n,\y\in\{-\frac{1}{\sqrt{m}},\frac{1}{\sqrt{m}}\}^m} \y^T H\x)}{\sqrt{n}} \leq (\sqrt{\alpha}+1)\sqrt{\frac{2}{\pi}}.\label{eq:posgenexplemma}
\end{equation}
\label{lemma:posgenlemma}
\end{lemma}
\begin{proof}
The proof of the first part follows from the above discussion. On the other hand, as mentioned above, the proof of the probability part follows in the same way as the corresponding one in \cite{StojnicHopBnds10}.
\end{proof}

The results that we presented in the above lemma are of course mathematically completely rigorous. However, they actually match the ones that can be obtained through the framework of the replica theory assuming the replica symmetry (see, e.g. \cite{BarGenGue11bip}). That essentially means that what we presented above actually establishes the replica symmetry prediction as a rigorous upper bound on the ground state energy of the asymmetric Little model.

To give some flavor as to what we presented so far we will fix $\alpha=1$. For such a scenario we essentially obtained a lower bound on the ground state energy that is equal to two times the ground state energy of the SK model, i.e. that is approximately equal to $2\times 0.763$. On the other hand we established an upper bound that is equal to the replica-symmetry prediction for the ground state energy, i.e. that is equal to $2\sqrt{\frac{2}{\pi}}\approx 2\times 0.7978$. We should also point out that the replica symmetry prediction for the ground state energy of the asymmetric Little model (i.e. $2\sqrt{\frac{2}{\pi}}\approx 2\times 0.7978$) is also two times the replica-symmetry prediction for the ground state energy of the SK model (which is $\sqrt{\frac{2}{\pi}}\approx \times 0.7978$). While all of this seems to be well aligned with the belief a physicist's approach to these problems would help form, it is still not clear what the real value of the ground state energy is. What we have shown indeed supports the idea that the ground state energy of the asymmetric Little model could be two times the ground state energy of the SK model but we haven't really excluded possibility that it may be higher than that as well.

In the following subsection we will provide a powerful mechanism to substantially lower the upper bound presented above. This may further substantiate belief that the original prediction about possible connection between the asymmetric Little and the SK model may in fact be correct. The mechanism that we will present will in large part rely on an important progress that we recently made while studying the Hopfield models in \cite{StojnicMoreSophHopBnds10}. The progress we made there enabled us to substantially lower the upper bound of the positive Hopfield form obtained in \cite{StojnicHopBnds10}. More importantly, it was the first example where we were able to move substantially further away from the typical replica symmetry bounds.

We should also take the opportunity to emphasize once again that the upper bounds for the positive Hopfield model presented in \cite{StojnicHopBnds10} as well as those for the asymmetric Little model presented above are essentially rigorous confirmation that the results of the replica symmetry theory are not only a good estimate but actually an estimate that is guaranteed to be from a certain side of the targeted quantity (in these cases, the replica symmetry theory estimates are upper bounds; however in some other scenarios they are often rigorous lower bounds; some examples include the negative Hopfield form discussed in \cite{StojnicHopBnds10} or the corresponding ones related to minmax version of the asymmetric Little form discussed below in Section \ref{sec:neghop} as well as various other results presented in e.g. \cite{StojnicCSetam09,StojnicRicBnds13,StojnicCSetamBlock09}).

\subsection{Lowering the upper-bound on the  ground state energy of the asymmetric Little form}
\label{sec:poshopublow}

The results that we presented in the previous subsection are a relatively easy way to establish upper-bounds on the ground state energy of the asymmetric Little model. In fact, mathematically speaking, what we did in the previous subsection is nothing too original. We essentially just recognized that certain problems of interest in theoretical physics or statistical mechanics can be easily fitted into well known mathematical frameworks. Of course, if one ignores the physics component then the random optimization problems that we studied so far are not unknown. Our arguments given above clearly rely on mathematical results of \cite{Slep62} and \cite{Gordon85}. In fact, not only do they rely on what was shown in \cite{Slep62,Gordon85}, their slight modifications were already worked out as simple examples in e.g. \cite{Gordon85}.

On the other hand the results that we will present in this subsection will turn out to be not only practically substantially better than the corresponding ones from the previous subsection but also they will contain a substantial conceptual mathematical improvement as well. As it will soon become clear, the changes that we will adopt below when compared to the previous subsection will visually appear as minimal. However, such visually minimal changes result in our opinion in a substantial breakthrough in studying various random combinatorial problems (see, e.g. \cite{StojnicHopBnds10,StojnicMoreSophHopBnds10}).

To start the presentation off we recall that in this subsection we of course again look at the optimization problem from (\ref{eq:posham1}). To have it at hand we also rewrite it once again
\begin{equation}
\xi_p=\max_{\x\in\{-\frac{1}{\sqrt{n}},\frac{1}{\sqrt{n}}\}^n,\y\in\{-\frac{1}{\sqrt{m}},\frac{1}{\sqrt{m}}\}^m}\y^T H\x.\label{eq:sqrtposham2}
\end{equation}
Since we will continue to assume that the elements of $H$ are i.i.d. standard normal random variables we will recall on another result that relates to to statistical properties of certain Gaussian processes.
\begin{theorem}(\cite{Gordon85})
\label{thm:Gordonpos1} Let $X_{i}$ and $Y_{i}$, $1\leq i\leq n$, be two centered Gaussian processes which satisfy the following inequalities for all choices of indices
\begin{enumerate}
\item $E(X_{i}^2)=E(Y_{i}^2)$
\item $E(X_{i}X_{l})\leq E(Y_{i}Y_{l}), i\neq l$.
\end{enumerate}
Let $\psi()$ be an increasing function on the real axis. Then
\begin{equation*}
E(\min_{i}\psi(X_{i}))\leq E(\min_i \psi(Y_{i})) \Leftrightarrow E(\max_{i}\psi(X_{i}))\geq E(\max_i\psi(Y_{i})).
\end{equation*}
\end{theorem}

In the previous subsection we relied on the above theorem to create an upper-bound on the ground state energy of the asymmetric Little model. However, the strategy employed in the previous subsection utilized only a basic version of the above theorem where $\psi(x)=x$. Here we will substantially upgrade the strategy by looking at a very simple (but way better) different version of $\psi()$. Along the same lines, we should mention that the strategy that we will present below will relate to what we presented in the previous subsection in the same way the results related to the positive Hopfield model obtained in \cite{StojnicMoreSophHopBnds10} relate to the corresponding ones from \cite{StojnicHopBnds10}.

We do recall without going into details that the ground state energies will concentrate in the thermodynamic limit. Hence in this subsection we will mostly focus on the expected value of $\xi_p$ (one can then easily adapt our results to describe more general probabilistic concentrating properties of ground state energies similarly to what was done in the previous subsection). The following is then a direct application of Theorem \ref{thm:Gordonpos1}.
\begin{lemma}
Let $H$ be an $m\times n$ matrix with i.i.d. standard normal components. Let $\g$ and $\h$ be $m\times 1$ and $n\times 1$ vectors, respectively, with i.i.d. standard normal components. Also, let $g$ be a standard normal random variable and let $c_3$ be a positive constant. Then
\begin{equation}
E(\max_{\x\in\{-\frac{1}{\sqrt{n}},\frac{1}{\sqrt{n}}\}^n,\|\y\|_2=1}e^{c_3(\y^T H\x + g)})\leq E(\max_{\x\in\{-\frac{1}{\sqrt{n}},\frac{1}{\sqrt{n}}\}^n,\|\y\|_2=1}e^{c_3(\g^T\y+\h^T\x)}).\label{eq:posexplemma}
\end{equation}\label{lemma:posexplemmalow}
\end{lemma}
\begin{proof}
As mentioned above, the proof is a standard/direct application of Theorem \ref{thm:Gordonpos1}. For a slightly different set of $\y$'s it was presented in \cite{StojnicMoreSophHopBnds10}. The structure of set of allowed $\y$'s changes nothing and the proof from \cite{StojnicMoreSophHopBnds10} can be translated step by step here. We skip such an easy exercise.
\end{proof}

To make the previous lemma operational we follow/adapt the strategy of \cite{StojnicMoreSophHopBnds10}. After repeating step by step what was done in \cite{StojnicMoreSophHopBnds10} one arrives at
\begin{equation}
\hspace{-.5in}E(\max_{\x\in\{-\frac{1}{\sqrt{n}},\frac{1}{\sqrt{n}}\}^n,\y\in\{-\frac{1}{\sqrt{m}},\frac{1}{\sqrt{m}}\}^m}\y^T H\x)\leq
-\frac{c_3}{2}+\frac{1}{c_3}\log(E(\max_{\x\in\{-\frac{1}{\sqrt{n}},\frac{1}{\sqrt{n}}\}^n}(e^{c_3\h^T\x})))
+\frac{1}{c_3}\log(E(\max_{\y\in\{-\frac{1}{\sqrt{m}},\frac{1}{\sqrt{m}}\}^m}(e^{c_3\g^T\y}))).\label{eq:chpos8}
\end{equation}
Let $c_3=c_3^{(s)}\sqrt{n}$ where $c_3^{(s)}$ is a constant independent of $n$. Then (\ref{eq:chpos8}) becomes
\begin{eqnarray}
\frac{E(\max_{\x\in\{-\frac{1}{\sqrt{n}},\frac{1}{\sqrt{n}}\}^n,\y\in\{-\frac{1}{\sqrt{m}},\frac{1}{\sqrt{m}}\}^m}\y^T H\x)}{\sqrt{n}}
& \leq &
-\frac{c_3^{(s)}}{2}+\frac{1}{nc_3^{(s)}}\log(E(\max_{\x\in\{-1,1\}^n}(e^{c_3^{(s)}\h^T\x})))\nonumber \\
& + &\frac{1}{nc_3^{(s)}}\log(E(\max_{\y\in\{-1,1\}^m}(e^{c_3^{(s)}\sqrt{\frac{n}{m}}\g^T\y}))).\label{eq:chpos91}
\end{eqnarray}
Transforming further we obtain
\begin{multline}
-\frac{c_3^{(s)}}{2}+\frac{1}{nc_3^{(s)}}\log(E(\max_{\x\in\{-1,1\}^n}(e^{c_3^{(s)}\h^T\x})))
+ \frac{1}{nc_3^{(s)}}\log(E(\max_{\y\in\{-1,1\}^m}(e^{c_3^{(s)}\sqrt{\frac{n}{m}}\g^T\y})))\\
=
-\frac{c_3^{(s)}}{2}+\frac{1}{c_3^{(s)}}\log(E(e^{c_3^{(s)}|\h_1|}))+\frac{\alpha}{c_3^{(s)}}\log(E(e^{c_3^{(s)}\sqrt{\frac{1}{\alpha}}|\g_1|}))
\\
 =
-\frac{c_3^{(s)}}{2}+\frac{c_3^{(s)}}{2}+\frac{1}{c_3^{(s)}}\log(\mbox{erfc}(-\frac{c_3^{(s)}}{\sqrt{2}}))
+\frac{c_3^{(s)}}{2}+\frac{\alpha}{c_3^{(s)}}\log(\mbox{erfc}(-\frac{c_3^{(s)}}{\sqrt{2\alpha}}))
 \\
 = \frac{c_3^{(s)}}{2}+\frac{1}{c_3^{(s)}}\log(\mbox{erfc}(-\frac{c_3^{(s)}}{\sqrt{2}}))
+\frac{\alpha}{c_3^{(s)}}\log(\mbox{erfc}(-\frac{c_3^{(s)}}{\sqrt{2\alpha}})).\label{eq:chpos92}
\end{multline}
Connecting (\ref{eq:chpos91}) and (\ref{eq:chpos92}) we finally have
\begin{equation}
\frac{E(\max_{\x\in\{-\frac{1}{\sqrt{n}},\frac{1}{\sqrt{n}}\}^n,\y\in\{-\frac{1}{\sqrt{m}},\frac{1}{\sqrt{m}}\}^m}\y^T H\x)}{\sqrt{n}}
\leq \frac{c_3^{(s)}}{2}+\frac{1}{c_3^{(s)}}\log(\mbox{erfc}(-\frac{c_3^{(s)}}{\sqrt{2}}))
+\frac{\alpha}{c_3^{(s)}}\log(\mbox{erfc}(-\frac{c_3^{(s)}}{\sqrt{2\alpha}})).\label{eq:chpos9}
\end{equation}
One should now note that the above bound is effectively correct for any positive constant $c_3^{(s)}$.  Of course to make it as tight as possible one then has
\begin{equation}
\hspace{-.1in}\frac{E(\max_{\x\in\{-\frac{1}{\sqrt{n}},\frac{1}{\sqrt{n}}\}^n,\y\in\{-\frac{1}{\sqrt{m}},\frac{1}{\sqrt{m}}\}^m}\y^T H\x)}{\sqrt{n}}
\leq \min_{c_3^{(s)}\geq 0}\left ( \frac{c_3^{(s)}}{2}+\frac{1}{c_3^{(s)}}\log(\mbox{erfc}(-\frac{c_3^{(s)}}{\sqrt{2}}))
+\frac{\alpha}{c_3^{(s)}}\log(\mbox{erfc}(-\frac{c_3^{(s)}}{\sqrt{2\alpha}}))\right ).\label{eq:ubmorsoph1}
\end{equation}

We summarize our results from this subsection in the following lemma.

\begin{lemma}
Let $H$ be an $m\times n$ matrix with i.i.d. standard normal components. Let $n$ be large and let $m=\alpha n$, where $\alpha>0$ is a constant independent of $n$. Let $\xi_p$ be as in (\ref{eq:sqrtposham1}). Let $\xi_p^{(u,low)}$ be a scalar such that
\begin{equation}
\xi_p^{(u,low)}=\min_{c_3^{(s)}\geq 0}\left ( \frac{c_3^{(s)}}{2}+\frac{1}{c_3^{(s)}}\log(\mbox{erfc}(-\frac{c_3^{(s)}}{\sqrt{2}}))
+\frac{\alpha}{c_3^{(s)}}\log(\mbox{erfc}(-\frac{c_3^{(s)}}{\sqrt{2\alpha}}))\right ).\label{eq:condxipuposgenlemma}
\end{equation}
Then
\begin{equation}
\frac{E\xi_p}{\sqrt{n}}\leq\xi_p^{(u,low)}.\label{eq:posgenexplemma}
\end{equation}
Moreover,
\begin{eqnarray}
& & \lim_{n\rightarrow\infty}P(\max_{\x\in\{-\frac{1}{\sqrt{n}},\frac{1}{\sqrt{n}}\}^n,\y\in\{-\frac{1}{\sqrt{m}},\frac{1}{\sqrt{m}}\}^m}\y^T H\x\leq \xi_p^{(u,low)})\geq 1\nonumber \\
& \Leftrightarrow & \lim_{n\rightarrow\infty}P(\xi_p\leq \xi_p^{(u,low)})\geq 1 \nonumber \\
& \Leftrightarrow & \lim_{n\rightarrow\infty}P(\xi_p^2\leq (\xi_p^{(u,low)})^2)\geq 1. \label{eq:posgenproblemma}
\end{eqnarray}
In particular, when $\alpha=1$
\begin{equation}
\frac{E\xi_p}{\sqrt{n}}\leq\xi_p^{(u,low)}=2\times 0.7688.\label{eq:posgenexplemma1}
\end{equation}
\label{lemma:posgenlemmalow}
\end{lemma}
\begin{proof}
The first part of the proof related to the expected values follows from the above discussion. The probability part follows by the concentration arguments from the previous subsections (also, see, e.g. discussion in \cite{StojnicHopBnds10,StojnicMoreSophHopBnds10}).
\end{proof}

One way to see how the above lemma works in practice is (as specified in the lemma) to choose $\alpha=1$ to obtain $\xi_p^{(u,low)}=2\times 0.7688$. This value is substantially better than $2\times 0.7978$ offered in the previous subsection. Also, we should finally clarify a bit what we mean by a substantial improvement. If one just merely looks at numbers, improvement from $0.7978$ to $0.7688$ may not seem as a whole lot. However, there are two things to keep in mind. First, according to the lower bounds presented in Subsection \ref{sec:poshoplb} the value we are discussing in any event can not go below $0.763$. So in terms of closing the gap from $0.763-0.7978$ to $0.763-0.7688$ the improvement may actually seem rather substantial. Second, while all the numbers we just mentioned seem close to each other, an extensive experience in studying random combinatorial problems usually suggests that any improvement over the standard replica-symmetry type of results is typically a big breakthrough. Here, not only that the strategy presented above manages to go below the replica-symmetry predictions (actually a rigorous upper bounds as shown in the previous subsection) but it almost approaches the lower bounding value. Moreover, all of that is achieved with a seemingly fairly simple yet super powerful approach. While there are other ways one can utilize to provide upper bounds on the ground state energy of the asymmetric Little model, we should mention that not that many of them we found capable of providing the values below the replica-symmetry predictions (upper bounds). Moreover, it is only the techniques that we presented in this and the previous two subsections that we observed as capable of making a substantial improvement over the replica-symmetry predictions (i.e. capable of reaching into a $.76...$ zone).

\section{Minmax asymmetric Little form}
\label{sec:neghop}

In this section we will look at the following optimization problem
\begin{equation}
\min_{\x\in\{-\frac{1}{\sqrt{n}},\frac{1}{\sqrt{n}}\}^n}\max_{\y\in\{-\frac{1}{\sqrt{m}},\frac{1}{\sqrt{m}}\}^m}\y^T H\x.\label{eq:negham1}
\end{equation}
Clearly, this is a minmax version of the problem given in (\ref{eq:posham1}) that we studied in the previous section. This problem may not have as much of importance within the theoretical physics community as its max counterpart (\ref{eq:posham1}) does have. However, when viewed in a purely optimization/algorithmic way it is a more general variant of a well known combinatorial problem called number partitioning problem (NPP). In fact, for $m=1$ in (\ref{eq:negham1}) one has exactly the random real number partitioning problem. For an integer $m>1$ one then has a multiple number partitioning problem where overall partitioning deviation is estimated through the $\ell_1$ norm (of course instead of the $\ell_1$ one can look at deviations with respect to say the $\ell_2$ norm; in such scenarios one then for $m>1$ has a variant of the negative Hopfield model; of course when $m=1$ the $\ell_1$ and $\ell_2$ deviations would be the same and one has that a single number partitioning problem can be casted as both, a variant of the negative Hopfield model and as a variant of minmax of the asymmetric Little model). We should also recall that in Section \ref{sec:back} we mentioned that throughout the paper we would assume that $m$ and $n$ are large and proportional to each other (although the proportionality constant, namely $\alpha$, can be assumed to be arbitrarily small). Still, $m=1$ would not fit into such a scenario. However, as we also mentioned in Section \ref{sec:back}, many of the results that we will present will be applicable even for fixed (finite) values of $m$ and $n$. Along the same line, the concepts that we will present below will remain useful even when $m=1$. However, some of the calculations that we will present below will have to be done a bit more carefully in that case and we will present such considerations in a separate paper that relates exclusively to the NPP's.

For a deterministic (given fixed) $H$ the above problem (\ref{eq:negham1}) is of course known to be NP-hard (as (\ref{eq:posham1}), it essentially falls into the class of binary bilinear optimization problems). Instead of looking at the problem in (\ref{eq:negham1}) in a deterministic way i.e. in a way that assumes that matrix $H$ is deterministic, we will adopt the strategy of the previous section and look at it in a statistical scenario. Also as in the previous section, we will assume that the elements of matrix $H$ are i.i.d. standard normals. We will then call the form (\ref{eq:negham1}) with Gaussian $H$, the Gaussian minmax asymmetric Little form. On the other hand, if needed, we will call the form (\ref{eq:negham1}) with Bernoulli $H$ (the case probably more of interest in neural networks/statistical physics), the Bernoulli minmax asymmetric Little form. In the remainder of this section we will look at possible ways to estimate the optimal value of the optimization problem in (\ref{eq:negham1}). In fact we will introduce a strategy similar to the one presented in the previous section to create a lower-bound on the optimal value of (\ref{eq:negham1}).

\subsection{Lower-bounding ground state energy of the minmax asymmetric Little form}
\label{sec:neghoplb}

As mentioned above, the subject of consideration in this section is the problem from (\ref{eq:negham1}). As earlier we will call its optimal objective value $\xi_n$, i.e.
\begin{equation}
\xi_n=\min_{\x\in\{-\frac{1}{\sqrt{n}},\frac{1}{\sqrt{n}}\}^n}\max_{\y\in\{-\frac{1}{\sqrt{m}},\frac{1}{\sqrt{m}}\}^m}\y^T H\x.\label{eq:sqrtnegham1}
\end{equation}
As hinted above, we will be interested in typical behavior of $\xi_n$ when the above problem is of statistical nature. Also, as mentioned above, for the statistical concreteness we will assume that the elements of $H$ are i.i.d. standard normal random variables. Of course, as earlier, as far as the validity of the the results that we will present below is concerned, there is really not that much of a restriction on the assumed statistics. On the other hand we believe that there is a substantial difference in the quality/simplicity of the presentation and for that reason we choose to work in the standard normal setup (of course, we do believe that the standard normal setup offers the opportunities for the most elegant presentations). To utilize such a setup we start by recalling on the following set of results from \cite{Gordon88} (they relate to statistical properties of certain Gaussian processes which will obviously be useful in dealing with an $H$ of that type).
\begin{theorem}(\cite{Gordon88})
\label{thm:Gordonmesh1} Let $X_{ij}$ and $Y_{ij}$, $1\leq i\leq n,1\leq j\leq m$, be two centered Gaussian processes which satisfy the following inequalities for all choices of indices
\begin{enumerate}
\item $E(X_{ij}^2)=E(Y_{ij}^2)$
\item $E(X_{ij}X_{ik})\geq E(Y_{ij}Y_{ik})$
\item $E(X_{ij}X_{lk})\leq E(Y_{ij}Y_{lk}), i\neq l$.
\end{enumerate}
Then
\begin{equation*}
P(\bigcap_{i}\bigcup_{j}(X_{ij}\geq \lambda_{ij}))\leq P(\bigcap_{i}\bigcup_{j}(Y_{ij}\geq \lambda_{ij})).
\end{equation*}
\end{theorem}
The following, more simpler, version of the above theorem relates to the expected values.
\begin{theorem}(\cite{Gordon88})
\label{thm:Gordonmesh2} Let $X_{ij}$ and $Y_{ij}$, $1\leq i\leq n,1\leq j\leq m$, be two centered Gaussian processes which satisfy the following inequalities for all choices of indices
\begin{enumerate}
\item $E(X_{ij}^2)=E(Y_{ij}^2)$
\item $E(X_{ij}X_{ik})\geq E(Y_{ij}Y_{ik})$
\item $E(X_{ij}X_{lk})\leq E(Y_{ij}Y_{lk}), i\neq l$.
\end{enumerate}
Then
\begin{equation*}
E(\min_{i}\max_{j}(X_{ij}))\leq E(\min_i\max_j(Y_{ij})).
\end{equation*}
\end{theorem}
 For $m=1$ in Theorems \ref{thm:Gordonmesh1} and \ref{thm:Gordonmesh2} they effectively simplify to Theorems \ref{thm:Slepian1} and \ref{thm:Slepian2}, respectively (Theorems \ref{thm:Slepian1} and \ref{thm:Slepian2} are often referred to as the Slepian's lemma (see, e.g. \cite{Slep62}). In fact, to be completely chronologically exact, the two above theorems actually extended the Slepian's lemma.

Now, to create a lower-bounding strategy for the minmax asymmetric Little form we will rely on Theorems \ref{thm:Gordonmesh1} and Theorem \ref{thm:Gordonmesh2}.
As in the previous section, we will first focus on the expected value of $\xi_n$ and then on its more general probabilistic properties. The following is then a direct application of Theorem \ref{thm:Gordonmesh2}.
\begin{lemma}
Let $H$ be an $m\times n$ matrix with i.i.d. standard normal components. Let $\g$ and $\h$ be $m\times 1$ and $n\times 1$ vectors, respectively, with i.i.d. standard normal components. Also, let $g$ be a standard normal random variable. Then
\begin{equation}
E(\min_{\x\in\{-\frac{1}{\sqrt{n}},\frac{1}{\sqrt{n}}\}^n}\max_{\y\in\{-\frac{1}{\sqrt{m}},\frac{1}{\sqrt{m}}\}^m}(\y^T H\x + g))\geq E(\max_{\x\in\{-\frac{1}{\sqrt{n}},\frac{1}{\sqrt{n}}\}^n}\max_{\y\in\{-\frac{1}{\sqrt{m}},\frac{1}{\sqrt{m}}\}^m}(\g^T\y+\h^T\x)).\label{eq:negexplemma}
\end{equation}\label{lemma:negexplemma}
\end{lemma}
\begin{proof}
As mentioned above, the proof is a standard/direct application of Theorem \ref{thm:Gordonmesh2}. A sketch of it in a fairly similar scenario can be found in \cite{StojnicHopBnds10}.
\end{proof}

Using results of Lemma \ref{lemma:negexplemma} we then have
\begin{multline}
E(\min_{\x\in\{-\frac{1}{\sqrt{n}},\frac{1}{\sqrt{n}}\}^n}\max_{\y\in\{-\frac{1}{\sqrt{m}},\frac{1}{\sqrt{m}}\}^m} \y^T H\x =E(\min_{\x\in\{-\frac{1}{\sqrt{n}},\frac{1}{\sqrt{n}}\}^n}\max_{\y\in\{-\frac{1}{\sqrt{m}},\frac{1}{\sqrt{m}}\}^m}(\y^T H\x +g))\\ \geq E(\min_{\x\in\{-\frac{1}{\sqrt{n}},\frac{1}{\sqrt{n}}\}^n}\max_{\y\in\{-\frac{1}{\sqrt{m}},\frac{1}{\sqrt{m}}\}^m}
(\g^T\y+\h^T\x))=E\sum_{i=1}^{m}|\g_i|-E\sum_{i=1}^{n}|\h_i|\geq \sqrt{\frac{2}{\pi}}\sqrt{m}-\sqrt{\frac{2}{\pi}}\sqrt{n}.\label{eq:neghopaftlemma2}
\end{multline}
Connecting beginning and end of (\ref{eq:neghopaftlemma2}) we finally have a lower bound on $E\xi_n$ from (\ref{eq:sqrtnegham1}), i.e.
\begin{equation}
E\xi_n=E(\min_{\x\in\{-\frac{1}{\sqrt{n}},\frac{1}{\sqrt{n}}\}^n}\max_{\y\in\{-\frac{1}{\sqrt{m}},\frac{1}{\sqrt{m}}\}^m} \y^T H\x) \geq
\sqrt{\frac{2}{\pi}}\sqrt{m}-\sqrt{\frac{2}{\pi}}\sqrt{n}=(\sqrt{\alpha}-1)\sqrt{\frac{2}{\pi}}\sqrt{n},\label{eq:neghopubexp}
\end{equation}
or in a scaled (possibly) more convenient form
\begin{equation}
\frac{E\xi_n}{\sqrt{n}}=\frac{E(\min_{\x\in\{-\frac{1}{\sqrt{n}},\frac{1}{\sqrt{n}}\}^n}\max_{\y\in\{-\frac{1}{\sqrt{m}},\frac{1}{\sqrt{m}}\}^m} \y^T H\x)}{\sqrt{n}} \geq (\sqrt{\alpha}-1)\sqrt{\frac{2}{\pi}}.\label{eq:neghopubexp1}
\end{equation}
Of course, the above result will be useful as long as the most right quantity is positive, i.e. as long as $\alpha>1$.

We now turn to deriving a more general probabilistic result related to $\xi_n$. We will do so through the following lemma.
\begin{lemma}
Let $H$ be an $m\times n$ matrix with i.i.d. standard normal components. Let $\g$ and $\h$ be $m\times 1$ and $n\times 1$ vectors, respectively, with i.i.d. standard normal components. Also, let $g$ be a standard normal random variable and let $\zeta_{\x}$ be a function of $\x$. Then
\begin{equation}
P(\min_{\x\in\{-\frac{1}{\sqrt{n}},\frac{1}{\sqrt{n}}\}^n}\max_{\y\in\{-\frac{1}{\sqrt{m}},\frac{1}{\sqrt{m}}\}^m}(\y^T A\x+g-\zeta_{\x})\geq 0)\geq
P(\min_{\x\in\{-\frac{1}{\sqrt{n}},\frac{1}{\sqrt{n}}\}^n}\max_{\y\in\{-\frac{1}{\sqrt{m}},\frac{1}{\sqrt{m}}\}^m}(\g^T\y+\h^T\x-\zeta_{\x})\geq 0).\label{eq:negproblemma}
\end{equation}\label{lemma:negproblemma}
\end{lemma}
\begin{proof}
The proof is basically the same as the proof of Lemma \ref{lemma:negexplemma}. The only difference is that instead of Theorem \ref{thm:Gordonmesh2} it relies on Theorem \ref{thm:Gordonmesh1}.
\end{proof}

Let $\zeta_{\x}=\epsilon_{5}^{(g)}\sqrt{n}+\xi_n^{(l)}$ with all $\epsilon$'s being arbitrarily small positive constants independent of $n$ and
\begin{eqnarray}
& & (1-\epsilon_{1}^{(m)})\sqrt{m}\sqrt{\frac{2}{\pi}}-(1+\epsilon_{1}^{(n)})\sqrt{n}\sqrt{\frac{2}{\pi}}-\epsilon_{5}^{(g)}\sqrt{n}>\xi_n^{(l)}\nonumber \\
& \Leftrightarrow & (1-\epsilon_{1}^{(m)})\sqrt{\alpha}\sqrt{\frac{2}{\pi}}-(1+\epsilon_{1}^{(n)})\sqrt{\frac{2}{\pi}}-\epsilon_{5}^{(g)}>\frac{\xi_n^{(l)}}{\sqrt{n}}.\label{eq:negcondxipu}
\end{eqnarray}
Following further what was done in \cite{StojnicHopBnds10} one then has
\begin{multline}
 \lim_{n\rightarrow\infty}P(\min_{\x\in\{-\frac{1}{\sqrt{n}},\frac{1}{\sqrt{n}}\}^n}\max_{\y\in\{-\frac{1}{\sqrt{m}},\frac{1}{\sqrt{m}}\}^m}\y^T H\x\geq \xi_n^{(l)})\\\geq \lim_{n\rightarrow\infty}P(\min_{\x\in\{-\frac{1}{\sqrt{n}},\frac{1}{\sqrt{n}}\}^n}\max_{\y\in\{-\frac{1}{\sqrt{m}},\frac{1}{\sqrt{m}}\}^m}
 (\g^T\y+\h^T\x-\epsilon_{5}^{(g)}\sqrt{n})\geq \xi_n^{(l)})\geq 1.\label{eq:leftnegprobanal3}
\end{multline}

We summarize our results from this subsection in the following lemma.

\begin{lemma}
Let $H$ be an $m\times n$ matrix with i.i.d. standard normal components. Let $n$ be large and let $m=\alpha n$, where $\alpha>0$ is a constant independent of $n$. Let $\xi_n$ be as in (\ref{eq:sqrtnegham1}). Let all $\epsilon$'s be arbitrarily small positive constants independent of $n$ and let $\xi_n^{(l)}$ be a scalar such that
\begin{equation}
(1-\epsilon_{1}^{(m)})\sqrt{\alpha}\sqrt{\frac{2}{\pi}}-(1+\epsilon_{1}^{(n)})\sqrt{\frac{2}{\pi}}
-\epsilon_{5}^{(g)}>\frac{\xi_n^{(l)}}{\sqrt{n}}.\label{eq:negcondxipuneggenlemma}
\end{equation}
Then
\begin{eqnarray}
& & \lim_{n\rightarrow\infty}P(\min_{\x\in\{-\frac{1}{\sqrt{n}},\frac{1}{\sqrt{n}}\}^n}\max_{\y\in\{-\frac{1}{\sqrt{m}},\frac{1}{\sqrt{m}}\}^m}
\y^T H\x\geq \xi_n^{(l)})\geq 1\nonumber \\
& \Leftrightarrow & \lim_{n\rightarrow\infty}P(\xi_n\geq \xi_n^{(l)})\geq 1 \nonumber \\
& \Leftrightarrow & \lim_{n\rightarrow\infty}P(\xi_n^2\geq (\xi_n^{(l)})^2)\geq 1, \label{eq:neggenproblemma}
\end{eqnarray}
and
\begin{equation}
\frac{E\xi_n}{\sqrt{n}}=\frac{E(\max_{\x\in\{-\frac{1}{\sqrt{n}},\frac{1}{\sqrt{n}}\}^n}\max_{\y\in\{-\frac{1}{\sqrt{m}},\frac{1}{\sqrt{m}}\}^m} \y^T H\x}{\sqrt{n}} \geq (\sqrt{\alpha}-1)\sqrt{\frac{2}{\pi}}.\label{eq:neggenexplemma}
\end{equation}
\label{lemma:neggenlemma}
\end{lemma}
\begin{proof}
The first part follows from the above discussion. The probability part follows from the considerations presented in \cite{StojnicHopBnds10}.
\end{proof}
As mentioned earlier, the results from the above lemma will be useful only when $\alpha>1$.

\subsection{Lifting the lower-bound on the  ground state energy of the minmax asymmetric Little form}
\label{sec:neghoplblift}

Similarly to what was the case when in Section \ref{sec:poshop} we studied the asymmetric Little form, we recall that the results we presented in the previous subsection are a relatively easy way to establish lower-bounds on the ground state energy of the minmax asymmetric Little model. Again from a pure mathematical prospective what we did in the previous subsection is nothing too original. We essentially just recognized that certain problems of interest in combinatorial optimization can be related to similar concepts from theoretical physics or statistical mechanics and then altogether can be easily fitted into well known probabilistic frameworks. As was the case in Section \ref{sec:poshop}, the random optimization problems that we studied so far in this section are not unknown. Our arguments given above clearly rely on mathematical results of \cite{Gordon85}. In fact, not only do they rely on what was shown in \cite{Gordon85}, their slight modifications were already worked out as simple examples in e.g. \cite{Gordon85}.

On the other hand the results that we will present in this subsection will turn out to be a conceptual counterpart to the improvement we provided for the max form in Subsection \ref{sec:poshopublow}. As such they will (perhaps not surprisingly any more) provide practically substantially better lower bounds than the methods of the previous subsection.

To start things off we recall that in this subsection we of course again look at the optimization problem from (\ref{eq:negham1}). To have it at hand we also rewrite it once again
\begin{equation}
\xi_n=\min_{\x\in\{-\frac{1}{\sqrt{n}},\frac{1}{\sqrt{n}}\}^n}\max_{\y\in\{-\frac{1}{\sqrt{m}},\frac{1}{\sqrt{m}}\}^m}\y^T H\x.\label{eq:sqrtposham2}
\end{equation}
Since we will continue to assume that the elements of $H$ are i.i.d. standard normal random variables we will recall on (and slightly extend) the following result from \cite{Gordon85} that relates to statistical properties of certain Gaussian processes. (This result is essentially a minmax counterpart to the one given in Theorem \ref{thm:Gordonpos1}.)
\begin{theorem}(\cite{Gordon85})
\label{thm:Gordonneg1} Let $X_{ij}$ and $Y_{ij}$, $1\leq i\leq n,1\leq j\leq m$, be two centered Gaussian processes which satisfy the following inequalities for all choices of indices
\begin{enumerate}
\item $E(X_{ij}^2)=E(Y_{ij}^2)$
\item $E(X_{ij}X_{ik})\geq E(Y_{ij}Y_{ik})$
\item $E(X_{ij}X_{lk})\leq E(Y_{ij}Y_{lk}), i\neq l$.
\end{enumerate}
Let $\psi()$ be an increasing function on the real axis. Then
\begin{equation*}
E(\min_{i}\max_{j}\psi(X_{ij}))\leq E(\min_{i}\max_{j}\psi(Y_{ij})).
\end{equation*}
Moreover, let $\psi()$ be a decreasing function on the real axis. Then
\begin{equation*}
E(\max_{i}\min_{j}\psi(X_{ij}))\geq E(\max_{i}\min_{j}\psi(Y_{ij})).
\end{equation*}
\begin{proof}
The proof of all statements but the last one is of course given in \cite{Gordon85}. A simple sketch of the validity of the last statement is given in \cite{StojnicMoreSophHopBnds10}.
\end{proof}
\end{theorem}

Although it may not be immediately visible, the results we presented in the previous subsection essentially also rely on the above theorem. As was the case with the max form in Section \ref{sec:poshop}, the strategy employed in the previous subsection relied only on a basic version of the above theorem where $\psi(x)=x$. Similarly to what was done in Subsection \ref{sec:poshopublow}, we will here substantially upgrade the strategy from the previous subsection by looking at a very simple (but way better) different version of $\psi()$.

As was the case with the max form in Subsection \ref{sec:poshopublow}, we do mention without going into details that the ground state energies will again concentrate in the thermodynamic limit and hence we will mostly focus on the expected value of $\xi_n$ (it is relatively easy to adapt our results to describe more general probabilistic concentrating properties of ground state energies similar to those presented in Subsections \ref{sec:poshoplb}, \ref{sec:poshopub}, and \ref{sec:neghoplb}). The following is then a direct application of Theorem \ref{thm:Gordonneg1}.
\begin{lemma}
Let $H$ be an $m\times n$ matrix with i.i.d. standard normal components. Let $\g$ and $\h$ be $m\times 1$ and $n\times 1$ vectors, respectively, with i.i.d. standard normal components. Also, let $g$ be a standard normal random variable and let $c_3$ be a positive constant. Then
\begin{equation}
E(\max_{\x\in\{-\frac{1}{\sqrt{n}},\frac{1}{\sqrt{n}}\}^n}\min_{\y\in\{-\frac{1}{\sqrt{m}},\frac{1}{\sqrt{m}}\}^m}e^{-c_3(\y^T H\x + g)})\leq E(\max_{\x\in\{-\frac{1}{\sqrt{n}},\frac{1}{\sqrt{n}}\}^n}\min_{\y\in\{-\frac{1}{\sqrt{m}},\frac{1}{\sqrt{m}}\}^m}e^{-c_3(\g^T\y+\h^T\x)}).\label{eq:negexplemma}
\end{equation}\label{lemma:negexplemmalift}
\end{lemma}
\begin{proof}
As mentioned above, the proof is a standard/direct application of Theorem \ref{thm:Gordonneg1}. A sketch of it in a similar scenario is given in \cite{StojnicMoreSophHopBnds10}.
\end{proof}

Following step by step what was done in \cite{StojnicMoreSophHopBnds10} one can establish a lower bound on the expected value of the ground state energy of the minmax asymmetric Little model
\begin{equation}
\hspace{-.5in}E(\min_{\x\in\{-\frac{1}{\sqrt{n}},\frac{1}{\sqrt{n}}\}^n}\max_{\y\in\{-\frac{1}{\sqrt{m}},\frac{1}{\sqrt{m}}\}^m}\y^T H\x\geq
\frac{c_3}{2}-\frac{1}{c_3}\log(E(\max_{\x\in\{-\frac{1}{\sqrt{n}},\frac{1}{\sqrt{n}}\}^n}(e^{-c_3\h^T\x})))
-\frac{1}{c_3}\log(E(\min_{\y\in\{-\frac{1}{\sqrt{m}},\frac{1}{\sqrt{m}}\}^m}(e^{-c_3\g^T\y}))).\label{eq:chneg8}
\end{equation}
As in Subsection \ref{sec:poshopublow}, let $c_3=c_3^{(s)}\sqrt{n}$ where $c_3^{(s)}$ is a constant independent of $n$. Then (\ref{eq:chneg8}) becomes
\begin{eqnarray}
\frac{E(\min_{\x\in\{-\frac{1}{\sqrt{n}},\frac{1}{\sqrt{n}}\}^n}\max_{\y\in\{-\frac{1}{\sqrt{m}},\frac{1}{\sqrt{m}}\}^m}\y^T H\x)}{\sqrt{n}}
& \geq &
\frac{c_3^{(s)}}{2}-\frac{1}{nc_3^{(s)}}\log(E(\max_{\x\in\{-1,1\}^n}(-e^{c_3^{(s)}\h^T\x})))\nonumber \\
& - &\frac{1}{nc_3^{(s)}}\log(E(\min_{\y\in\{-1,1\}^m}(e^{-c_3^{(s)}\sqrt{\frac{n}{m}}\g^T\y}))).\label{eq:chneg91}
\end{eqnarray}
Transforming further we obtain
\begin{multline}
\frac{c_3^{(s)}}{2}-\frac{1}{nc_3^{(s)}}\log(E(\max_{\x\in\{-1,1\}^n}(e^{-c_3^{(s)}\h^T\x})))
- \frac{1}{nc_3^{(s)}}\log(E(\min_{\y\in\{-1,1\}^m}(e^{-c_3^{(s)}\sqrt{\frac{n}{m}}\g^T\y})))\\
=
\frac{c_3^{(s)}}{2}-\frac{1}{c_3^{(s)}}\log(E(e^{c_3^{(s)}|\h_1|}))-\frac{\alpha}{c_3^{(s)}}\log(E(e^{-c_3^{(s)}\sqrt{\frac{1}{\alpha}}|\g_1|}))
\\
 =
\frac{c_3^{(s)}}{2}-\frac{c_3^{(s)}}{2}-\frac{1}{c_3^{(s)}}\log(\mbox{erfc}(-\frac{c_3^{(s)}}{\sqrt{2}}))
-\frac{c_3^{(s)}}{2}-\frac{\alpha}{c_3^{(s)}}\log(\mbox{erfc}(\frac{c_3^{(s)}}{\sqrt{2\alpha}}))
 \\
 =-\frac{c_3^{(s)}}{2}-\frac{1}{c_3^{(s)}}\log(\mbox{erfc}(-\frac{c_3^{(s)}}{\sqrt{2}}))
-\frac{\alpha}{c_3^{(s)}}\log(\mbox{erfc}(\frac{c_3^{(s)}}{\sqrt{2\alpha}})).\label{eq:chneg92}
\end{multline}
Connecting (\ref{eq:chneg91}) and (\ref{eq:chneg92}) we finally have
\begin{equation}
\frac{E(\min_{\x\in\{-\frac{1}{\sqrt{n}},\frac{1}{\sqrt{n}}\}^n}\max_{\y\in\{-\frac{1}{\sqrt{m}},\frac{1}{\sqrt{m}}\}^m}\y^T H\x)}{\sqrt{n}}
\geq -\frac{c_3^{(s)}}{2}-\frac{1}{c_3^{(s)}}\log(\mbox{erfc}(-\frac{c_3^{(s)}}{\sqrt{2}}))
-\frac{\alpha}{c_3^{(s)}}\log(\mbox{erfc}(\frac{c_3^{(s)}}{\sqrt{2\alpha}})).\label{eq:chneg9}
\end{equation}
One should again note that the above bound is effectively correct for any positive constant $c_3^{(s)}$.  Of course to make it as tight as possible one then has
\begin{equation}
\hspace{-.1in}\frac{E(\min_{\x\in\{-\frac{1}{\sqrt{n}},\frac{1}{\sqrt{n}}\}^n}\max_{\y\in\{-\frac{1}{\sqrt{m}},\frac{1}{\sqrt{m}}\}^m}\y^T H\x)}{\sqrt{n}}
\geq \max_{c_3^{(s)}\geq 0}\left (-\frac{c_3^{(s)}}{2}-\frac{1}{c_3^{(s)}}\log(\mbox{erfc}(-\frac{c_3^{(s)}}{\sqrt{2}}))
-\frac{\alpha}{c_3^{(s)}}\log(\mbox{erfc}(\frac{c_3^{(s)}}{\sqrt{2\alpha}}))\right ).\label{eq:lbmorsoph1}
\end{equation}

We summarize our results from this subsection in the following lemma.

\begin{lemma}
Let $H$ be an $m\times n$ matrix with i.i.d. standard normal components. Let $n$ be large and let $m=\alpha n$, where $\alpha>0$ is a constant independent of $n$. Let $\xi_n$ be as in (\ref{eq:sqrtnegham1}) let and $\xi_n^{(l,lift)}$ be a scalar such that
\begin{equation}
\xi_n^{(l,lift)}=\max_{c_3^{(s)}\geq 0}\left (-\frac{c_3^{(s)}}{2}-\frac{1}{c_3^{(s)}}\log(\mbox{erfc}(-\frac{c_3^{(s)}}{\sqrt{2}}))
-\frac{\alpha}{c_3^{(s)}}\log(\mbox{erfc}(\frac{c_3^{(s)}}{\sqrt{2\alpha}}))\right ).\label{eq:condxipuneggenlemma}
\end{equation}
Then
\begin{equation}
\frac{E\xi_n}{\sqrt{n}}\geq\xi_n^{(l,lift)}.\label{eq:neggenexplemma}
\end{equation}
Moreover,
\begin{eqnarray}
& & \lim_{n\rightarrow\infty}P(\min_{\x\in\{-\frac{1}{\sqrt{n}},\frac{1}{\sqrt{n}}\}^n}\max_{\y\in\{-\frac{1}{\sqrt{m}},\frac{1}{\sqrt{m}}\}^m}
\y^T H\x\geq \xi_n^{(l,lift)})\geq 1\nonumber \\
& \Leftrightarrow & \lim_{n\rightarrow\infty}P(\xi_n\geq \xi_n^{(l,lift)})\geq 1 \nonumber \\
& \Leftrightarrow & \lim_{n\rightarrow\infty}P(\xi_n^2\geq (\xi_n^{(l,lift)})^2)\geq 1. \label{eq:neggenproblemma}
\end{eqnarray}
In particular, when $\alpha=1$
\begin{equation}
\frac{E\xi_n}{\sqrt{n}}\geq\xi_n^{(l,lift)}=0.24439.\label{eq:neggenexplemma1}
\end{equation}
\label{lemma:neggenlemmalift}
\end{lemma}
\begin{proof}
The first part of the proof related to the expected values follows from the above discussion. The probability part follows by the concentration arguments similar to those we presented in Section \ref{sec:poshop} (also, see, e.g discussion in \cite{StojnicHopBnds10}).
\end{proof}

Similarly to what we did in Subsection \ref{sec:poshopublow}, one can assess how the above lemma works in practice by chooosing a particular $\alpha$. Also, as we did in is Subsection \ref{sec:poshopublow} (and as we specified above in Lemma \ref{lemma:neggenlemmalift}) we selected a popular squared choice, i.e. we selected $\alpha=1$. For such an $\alpha$ one obtains $\xi_n^{(l,lift)}=0.24439$. This value is of course substantially better than $0$ offered by considerations presented in the previous subsection. Also, one should note that solving minmax version of the asymmetric Little problem can be much harder than solving the max version considered in Section \ref{sec:poshop}. According to the results that we presented above this seems certainly true if the main concern is the typical behavior of optimal values of the problem's random instances (moreover, the minmax problem is substantially harder algorithmically as well). In fact the analytical results that we provided in this paper actually indicate what seems to be a trend. Namely, the upper bound that we obtained on the max version of the problem in Subsection \ref{sec:poshopub} was conceptually substantially improved through the mechanism of Subsection \ref{sec:poshopublow}. However, practically/numerically the improvement is hardly a few percents. On the other hand the lower bound that we obtained on the minmax version of the problem in Subsection \ref{sec:neghoplb} was both, conceptually and practically, substantially improved through the mechanism of Subsection \ref{sec:neghoplblift}. A fairly similar trend we observed when we studied the Hopfield models in \cite{StojnicHopBnds10,StojnicMoreSophHopBnds10}. The upper bounds on the max version of the problem were improved conceptually but the practical improvement was not that substantial. On the other hand, in addition to the conceptual improvement, the numerical improvement on the lower bounds of the minmax version of the Hopfield problem was also substantial. Moreover, if one compares the max and minmax problems for the asymmetric Little and Hopfield model the improvement we achieved for the Little model cases is numerically more significant. All of this of course suggests/confirms that the replica symmetry predictions are much closer to the true values when it comes to the max problems.

\section{Conclusion}
\label{sec:conc}

In this paper we looked at classic asymmetric Little models from statistical mechanics. These models were thoroughly studied both numerically and theoretically in e.g. \cite{BruParRit92,CabMarPaoPar88,BarGenGue11bip}. The replica theory investigations conducted in \cite{BruParRit92} essentially predicted that the asymmetric Little model should resemble the well-known/studied Sherrington-Kirkpatrick (SK) model. We in this paper provided a few mathematically rigorous results that hint that the predictions made in \cite{BruParRit92} and (to a degree later in \cite{BarGenGue11bip}) may indeed be correct.

More specifically, we studied the behavior of the free energy in the thermodynamic limit at zero temperature (this of course corresponds to what is in statistical physics typically called the ground state energy in the thermodynamic limit). We proved the lower bounds on the ground state energy that match the appropriately scaled values one would get for the SK model. We also created a simple upper bound which turned out to match the one that could be obtained through the replica theory approach assuming the replica symmetry (see, e.g. \cite{BarGenGue11bip,BruParRit92}). Moreover, we then presented a powerful mechanism that can be used to lower  this simple upper bound. Finally, the lowered upper bound came fairly close to the the lower bound which matches the corresponding one of the appropriately scaled SK model. Consequently, the results that we provided hint that the original predictions about resemblance/equivalnce between the asymmetric Little and the appropriately scaled SK model may in fact be mathematically correct.

For the sake of simplicity of the exposition we chose to focus on the free energy in the zero temperature limit. All our considerations can easily be extended to cover the behavior of the free energy at any temperature. That requires no further insight but does require a substantially more detailed presentation and we will defer further discussion in that direction to a forthcoming paper.

We then switched from the standard asymmetric Little model to what we referred to as the minmax asymmetric Little model/form (in such a context we often referred to the standard asymmetric Little model as the max asymmetric Little model). While the minmax model on its own may not have as much of statistical physics prominence as does the max one it is nevertheless an interesting/important mathematical/optimization type of problem. Namely, the resulting problem is a version of what is called number partitioning problem in optimization/complexity/algorithms theory. We then studied the typical behavior of this problem and presented a set of results that resemble in large part the ones that we presented for the max case. Again, the improvements we achieved over standard methods seemed both, conceptually and practically significant.

Finally, we should mention that we presented a collection of theoretical results for a particular type of randomness, namely the standard normal one. However, as was the case when we studied the Hopfield models in \cite{StojnicHopBnds10,StojnicMoreSophHopBnds10}, all results that we presented can easily be extended to cover a wide range of other types of randomness. There are many ways how this can be done (and the rigorous proofs are not that hard either). Typically they all would boil down to a repetitive use of the central limit theorem. For example, a particularly simple and elegant approach would be the one of Lindeberg \cite{Lindeberg22}. Adapting our exposition to fit into the framework of the Lindeberg principle is relatively easy and in fact if one uses the elegant approach of \cite{Chatterjee06} pretty much a routine. However, as we mentioned when studying the Hopfield model \cite{StojnicHopBnds10}, since we did not create these techniques we chose not to do these routine generalizations. On the other hand, to make sure that the interested reader has a full grasp of a generality of the results presented here, we do emphasize again that pretty much any distribution that can be pushed through the Lindeberg principle would work in place of the Gaussian one that we used.

It is also important to emphasize that we in this paper presented a collection of very simple observations. One can improve many of the results that we presented here but at the expense of the introduction of a more complicated theory. We will present results in such a direction elsewhere.

\begin{singlespace}
\bibliographystyle{plain}
\bibliography{AsymmLittBndsRefs}
\end{singlespace}

\end{document}